\documentclass{amsart}
\usepackage{amsmath}
\usepackage{amssymb}
\usepackage{amsthm}
\usepackage{color}
\usepackage{dsfont}
\usepackage{yhmath}
\usepackage{mathrsfs}
\usepackage{enumerate}

\newcommand{\eps}{\varepsilon}
\newtheorem{Def}{Definition}[section]
\newtheorem{Prop}{Proposition}[section]
\newcommand{\R}{\mathbb{R}}
\newcommand{\chap}{\widehat}
\newcommand{\daw}{\leqslant}

\newcommand{\lang}{\left\langle}
\newcommand{\rang}{\right\rangle}
\newcommand{\N}{\mathbb{N}}
\newcommand{\C}{\mathbb{C}}
\numberwithin{equation}{section}
\theoremstyle{plain}
\newtheorem{The}{Theorem}[section]
\newtheorem{lemme}{Lemma}[section]
\theoremstyle{definition}

\begin{document}
\title[The semi-classical singular Hartree equation]{\textbf{High-frequency averaging in the semi-classical singular Hartree equation}}
\author[Lounes MOUZAOUI]{Lounes MOUZAOUI}
\address{Univ. Montpellier~2\\Math\'ematiques, UMR 5149
\\CC~051\\34095 Montpellier\\ France}
\email{lounes.mouzaoui@univ-montp2.fr}
\date{\today}
\thanks{This work was supported by the French ANR project
  R.A.S. (ANR-08-JCJC-0124-01).}							

\begin{abstract} We study the asymptotic behavior of the Schr\"odinger equation in the
presence of a nonlinearity of Hartree type in the semi-classical regime. Our scaling corresponds
to a weakly nonlinear regime where the nonlinearity affects the leading order amplitude of the
solution without altering the rapid oscillations. We show the validity of the WKB-analysis when
the potential in the nonlinearity is singular around the origin. No new resonant wave is created in our model, this phenomenon is
inhibited due to the nonlinearity. The nonlocal nature of this latter leads us to build our result on a high-frequency averaging effects. In the proof we make use of the Wiener algebra and the space of square-integrable functions.
\end{abstract}
\maketitle

\section{Introduction}
In this paper we are interested in studying of the Schr\"odinger equation
in the presence of a \emph{Hartree type nonlinearity}

\begin{equation}
\label{l'equa2}
i\eps\partial_{t} u^{\eps} +
\frac{\eps^{2}}{2} \Delta u^{\eps} =
\eps\lambda (K*{|u^{\eps}|}^{2})u^{\eps},
\end{equation}
with $\eps >0$, $u^{\eps}\in \mathcal{S}'(I \times \mathbb{R}^{d},
\mathbb{C}), d \in \{1,2,3\}, I $ interval of $\mathbb{R}$, $\lambda \in \R, \mbox{and } K(x) = \frac{1}{|x|^{\gamma}}$
with $0<\gamma<d$. This model
is of physical application.
It appears in many physical phenomena like in describing superfluids (see \cite{Berloff99}).
The presence of $\eps$ in this equation is to show the microscopic/macroscopic scale ratio.	For small $\eps >0$, the scaling
of $\eqref{l'equa2}$ corresponds to the \textit{semi-classical}
regime, i.e. the regime of \textit{high-frequency}
solutions $u^{\eps}(t,x)$. This system contains the Schr\"odinger-Poisson system ($d=3, \gamma=1$),
a model which is used in studying the quantum transport in semi-conductor devices (see \cite{MR1387456, semi-cond-ref, Sconducteur2}). Many authors have interested to study the asymptotic behavior of \eqref{l'equa2}
in this case like in \cite{{MR1251718}, MR1203274, art18} where it is made use of Wigner measure techniques.
Several papers like \cite{MR2290279, art7}  are devoted to study the strong asymptotic limit of the solutions of \eqref{l'equa2} with Hartree nonlinearity in the case of a \emph{single-phase} WKB initial data in the form
\[
u^{\eps}_{0}(x)=\alpha^{\eps}(x)e^{i\varphi(x)/\eps},
\]
with $\varphi(x)\in \C$  $\eps$-independent and $\alpha^{\eps}(x) \in \R.$

Due to the small parameter $\eps$ in
front of the nonlinearity, we consider a
\textit{weakly nonlinear regime.} This means that the nonlinearity
does not affect the geometry of the propagation (see Section 3), and technically
means that it does not show up in the eikonal equation, but only in the transport
equations determining the amplitudes.
The object of this paper is to construct an
approximate solution $u^{\varepsilon}_{app}$ for the exact solution
$u^{\varepsilon}$ of \eqref{l'equa2} subject to an initial data of WKB type, given by a superposition of
$\eps$-oscillatory plane waves, i.e.

\begin{equation}
\label{cond_init}
\left.u^{\eps}_{app}\right|_{t=0}(x)=u^{\eps}_{0}(x)=\sum_{j\in \N} \alpha_{j}(x)e^{i\kappa_{j}\cdot x/\varepsilon}.
\end{equation}

We begin the paper by showing the existence of an exact solution to \eqref{l'equa2} in $L^{2}(\R^{d})\cap W(\R^{d})$ where
$W(\R^{d})$ denotes the \emph{Wiener algebra}
\[
W(\R^{d}) = \mathcal{F}(L^{1}(\R^{d}))=\lbrace f \in \mathcal{S}'(\R^{d},\C), \|f\|_{W}:=\|\chap{f}\|_{L^{1}(\R^{d})} < +\infty\rbrace,
\]
and where
\[
(\mathcal{F}f)(\xi)=\chap{f}(\xi)=\frac{1}{(2\pi)^{d/2}} \int_{\R^{d}} f(x)e^{-ix\cdot \xi} \, \mathrm{d}x.
\]
In the next step,
we construct an approximate	solution in the form
\[
u^{\eps}_{app}(t,x) = \sum_{j\in \N} a_{j}(t,x)e^{i\phi_{j}(t,x)/\eps},
\]
in $L^{2}(\R^{d})\cap W(\R^{d})$
by determining the
amplitudes $a_{j}$ and the phases $\phi_{j}$,
then we proceed to study its stability to justify our
construction.

Fourier transform of the potential $K$ is found to be
\[
\chap{K}(\xi)=\frac{C_{d,\gamma}}{|\xi|^{d-\gamma}},
\]
(see [3, Proposition 1.29]) so under the general assumptions of \cite{MR2731651} where the considered kernel $K$ is such that
$(1+|\xi|)\chap{K}(\xi) \in L^{\infty}(\R^{d})$ we can not deduce the well posedness of \eqref{l'equa2}.
It is a critical case for our model because we do not know how to construct an exact local solution in $W(\R^{d})$
due to the singularity of
$\chap{K}$ around the origin $(\chap{K} \notin L^{\infty})$.

To define the framework of the amplitudes $(a_{j})_{j\in \N}$
we introduce the following definitions.
\begin{Def}
\label{def_espace}
For $d \in \lbrace 1,2,3 \rbrace$ and $0< \gamma < d  $ we define $n\in\N$ as follows
\[
n = \left\{
\begin{array}{ll}
2 & \mbox{if }\, d=1 \mbox{ or } 2 \\
2 & \mbox{if }\, d=3 \mbox{ and } \gamma \in [1,3[\\
3 & \mbox{if }\, d=3 \mbox{ and } \gamma \in ]0,1[.\\
\end{array}
\right.
\]
We define the space
\[
Y(\R^{d}):=\lbrace f \in L^{2}\cap W, \partial^{\eta}f \in L^{2}\cap W,
\forall \eta \in \N^{d}, |\eta| \daw n\rbrace,
\]
equipped with the norm
\[
\|f\|_{Y(\R^{d})} = \sum_{0\daw |\eta| \daw n} \|\partial^{\eta}f\|_{L^{2}\cap W}.
\]
We set
\[
E(\R^{d})=\lbrace a=(a_{j})_{j\in \N}\; | \;(a_{j})_{j\in \N} \in \ell^{1}(\N, Y)\rbrace,
\]
which is a Banach space when it is equipped with the norm
\[
\|a\|_{E(\R^{d})}=\sum_{j\in \N} \|a_{j}\|_{Y}.
\]

\end{Def}

\begin{flushleft}
Now, we can state the main theorem of this work.
\end{flushleft}	
\begin{The}
\label{mainTh}
For $d \in \lbrace 1,2,3\rbrace$ and $0<\gamma <d$, consider the Cauchy problem \eqref{l'equa2},
subject to initial data $u^{\eps}_{0}$ of the form \eqref{cond_init}, with initial amplitudes
$(\alpha_{j})_{j\in \N} \in E(\R^{d})$. We assume that there
exists $\delta >0$ such that

\[
\inf\lbrace|\kappa_{k}-\kappa_{m}|, k,m
\in \N, k \neq m\rbrace \geqslant \delta > 0.
\]

Then, for all $T >0$ there exists $C, \eps_{0}(T) >0$, such that for any $ \eps \in ]0, \eps_{0}]$,
the exact solution to \eqref{l'equa2} satisfies $u^{\eps} \in L^{\infty}([0,T];L^{2}\cap W)$ and can be approximated by

\[
\|u^{\eps}- u^{\eps}_{app}\|_{L^{\infty}([0,T];L^{2}\cap W)} \daw C\eps^{\beta},
\]
where $ \beta=\min\lbrace 1,d-\gamma \rbrace $ and where $u^{\eps}_{app}$ is defined by
\[
u^{\eps}_{app}(t,x)=\sum_{j\in \N} \alpha_{j}(x-t\kappa_{j})e^{iS_{j}(t,x)}e^{i\kappa_{j}\cdot x /\eps - i|\kappa_{j}|^{2}/ 2\eps},
\]
with $S_{j} \in \R$ defined in \eqref{ampli_puissance}.

\end{The}

In particular, for $d=3$ and $\gamma=1$ we obtain an approximation result for Schr\"odinger-Poisson equation with $\beta=1$.
This confirms in a general way what was guessed in \cite{{2010arXiv1006.4701C}} for the Schr\"odinger-Poisson system.
An important feature of our justification consists in the absence of phenomena
of phase resonances like we can see later.
\section{Existence and uniqueness}

\begin{flushleft}
We consider the following Cauchy problem

\begin{equation}
\label{eq:(SP')}
i\partial_{t} u +
\frac{1}{2}\Delta u =
\lambda (K*{|u|}^{2})u,
\end{equation}
subject to an initial data $u_{0} \in L^{2}(\R^{d}) \cap W(\R^{d})$,
with $K(x)= \frac{1}{|x|^{\gamma}}, 0<\gamma <d$. For $0 <\gamma<d$, the Fourier transform of $K$ for $\xi \in \R^{d}$ is
\[
\chap{K}(\xi)= \frac{C_{d,\gamma}}{|\xi|^{d-\gamma}}.
\]
In the following we denote $\mathcal{K}_{1}= \chap{K}(\xi) \mathds{1}_{[|\xi| \daw 1]} $
and $\mathcal{K}_{2}= \chap{K}(\xi) \mathds{1}_{[|\xi| > 1]}$, so, $\chap{K}=\mathcal{K}_{1}+\mathcal{K}_{2}$
with $\mathcal{K}_{1} \in L^{p}(\R^{d})$ for all $p\in [1,\frac{d}{d-\gamma}[$ and
$\mathcal{K}_{2} \in L^{q}(\R^{d})$ for all $q \in ]\frac{d}{d-\gamma}, +\infty]$. Let $g(u)=(K*|u|^{2})u$. For
$\eps >0$ we set
\[
U^{\eps}(t)=e^{i\eps\frac{t}{2} \Delta},
\]
with $U(t):=U^{1}(t)$. We recall some important properties of $W(\R^{d}).$

\begin{lemme}
\label{W_algebra}
Wiener algebra space $W(\R^{d})$ enjoys the following properties (see \cite{MR2607351, MR2515782}):
\begin{enumerate}[i.]
\item $W(\R^{d})$ is a Banach space, continuously embedded into $L^{\infty}(\R^{d})$.

\item $W(\R^{d})$ is an algebra, in the sense that the mapping $ (f,g)\mapsto fg$ is
continuous from $W(\R^{d})^{2}$ to $ W(\R^{d})$, and moreover
\[
\forall f,g \in W(\R^{d}), \; \|fg\|_{W} \daw \|f\|_{W}\|g\|_{W}.
\]

\item For all $t \in \R$, $U^{\eps}(t)$ is unitary on $W(\R^{d})$.

\item For all $s > \frac{d}{2}$ there exists a positive constant $C(s,d)$
such that for all $f\in H^{s}(\R^{d}) $
\[
\|f\|_{W} \daw C(s,d) \|f\|_{H^{s}}.
\]

\end{enumerate}
\end{lemme}
The following theorem ensures
the existence and uniqueness of the solution to \eqref{eq:(SP')}:

\begin{The}
\label{exist_uniq}
We consider the above initial value problem \eqref{eq:(SP')}
with $u_{0} \in L^{2}(\R^{d}) \cap W(\R^{d})$. Then there exists $T >0$ depending on $\|u_{0}\|_{L^{2}\cap W}$ and a unique
solution
$u \in C([0,T]; L^{2}(\R^{d})\cap W(\R^{d}))$ to \eqref{eq:(SP')}.
\end{The}

\begin{proof}

The following lemma will be useful to prove the above theorem :

\begin{lemme}
\label{lemm_W}
Let $h \in L^{1}\cap W, f_{1}, f_{2} \in L^{2}\cap W $ and $K$ like defined as above. Then

\begin{equation}
\label{ineg4}
\|K*h\|_{W} \daw \|\mathcal{K}_{1}\|_{L^{1}} \|h\|_{L^{1}} + \|\mathcal{K}_{2}\|_{L^{\infty}} \|h\|_{W},
\end{equation}

in particular we have
\[
\label{inge4}
\|g(f_{1})-g(f_{2})\|_{L^{2}\cap W} \daw C (\|f_{1}\|^{2}_{L^{2}\cap W}+
\|f_{2}\|^{2}_{L^{2}\cap W})\|f_{1}-f_{2}\|_{L^{2}\cap W},
\]
for some positive constant $C$.

\end{lemme}
\begin{proof}
By definition of the $W$ norm :
\[
\begin{aligned}
\|K*h\|_{W} &=C\|(\mathcal{K}_{1}+\mathcal{K}_{2}) \chap{h}\|_{L^{1}}\\
& \daw C\|\mathcal{K}_{1}\|_{L^{1}} \|\widehat{h}\|_{L^{\infty}} + C\|\mathcal{K}_{2}\|_{L^{\infty}} \|\widehat{h}\|_{L^{1}}\\
&\daw C\|\mathcal{K}_{1}\|_{L^{1}} \|h\|_{L^{1}} + C\|\mathcal{K}_{2}\|_{L^{\infty}} \|h\|_{W}.
\end{aligned}
\]
Moreover, for $f_{1}, f_{2} \in L^{2}\cap W$
\[
\|g(f_{1})-g(f_{2})\|_{L^{2}\cap W}= \phantom{\|(K*(|f_{1}|^{2}-|f_{2}|^{2}))f_{1}\|_{L^{2}\cap W}
+\|(K*|f_{2}|^{2})(f_{1}-f_{2})\|_{L^{2}\cap W}.}
\]
\[
\begin{aligned}
&=\|(K*(|f_{1}|^{2}-|f_{2}|^{2}))f_{1}+(K*|f_{2}|^{2})(f_{1}-f_{2})\|_{L^{2}\cap W}\\
&\daw\|(K*(|f_{1}|^{2}-|f_{2}|^{2}))f_{1}\|_{L^{2}\cap W}+\|(K*|f_{2}|^{2})(f_{1}-f_{2})\|_{L^{2}\cap W}.
\end{aligned}
\]
We control the last two norms. We have
\[
\|(K*|f_{2}|^{2})(f_{1}-f_{2})\|_{L^{2}\cap W}=\phantom{\|(K*|f_{2}|^{2})(f_{1}-f_{2})\|_{L^{2}}+ \|(K*|f_{2}|^{2})(f_{1}-f_{2})\|_{W}}
\]
\[
\begin{aligned}
&=\|(K*|f_{2}|^{2})(f_{1}-f_{2})\|_{L^{2}}+\|(K*|f_{2}|^{2})(f_{1}-f_{2})\|_{W}\\
&\daw\|K*|f_{2}|^{2}\|_{L^{\infty}}\|f_{1}-f_{2}\|_{L^{2}}\|K*|f_{2}|^{2}\|_{W}\|f_{1}-f_{2}\|_{W} \\
&\daw C \|K*|f_{2}|^{2}\|_{W}\|f_{1}-f_{2}\|_{L^{2}\cap W}\\
&\daw C \|\mathcal{K}_{1}\|_{L^{1}} \|f_{2}\|^{2}_{L^{2}}\|f_{1}-f_{2}\|_{L^{2}\cap W} +C\|\mathcal{K}_{2}\|_{L^{\infty}} \|f_{2}\|^{2}_{W}\|f_{1}-f_{2}\|_{L^{2}\cap W},\\
\end{aligned}
\]
where we have used \eqref{ineg4} and the embedding of $W(\R^{d})$ into $L^{\infty}(\R^{d})$. Remark that

\[
|f_{1}|^{2}-|f_{2}|^{2}= \frac{1}{2} (f_{1}-f_{2})(\overline{f_{1}}+\overline{f_{2}})+
\frac{1}{2} (\overline{f_{1}}-\overline{f_{2}})(f_{1}+f_{2}).
\]
From \eqref{ineg4} and Lemma \ref{W_algebra} iii, we have
\newpage
\[
\|(K*(|f_{1}|^{2}-|f_{2}|^{2}))f_{1}\|_{L^{2}\cap W} \daw \phantom{\|K*(|f_{1}|^{2}-|f_{2}|^{2})\|_{L^{\infty}}\|f_{1}\|_{L^{2}\cap W}
+\|K*(|f_{1}|^{2}-|f_{2}|^{2})\|_{W}\|f_{1}\|_{L^{2}\cap W}}
\]
\[
\begin{aligned}
&\daw\|K*(|f_{1}|^{2}-|f_{2}|^{2})\|_{L^{\infty}}\|f_{1}\|_{L^{2}\cap W}
+\|K*(|f_{1}|^{2}-|f_{2}|^{2})\|_{W}\|f_{1}\|_{L^{2}\cap W}\\
&\daw  C\|K*(|f_{1}|^{2}-|f_{2}|^{2})\|_{W}\|f_{1}\|_{L^{2}\cap W} \\
&\daw C\|K*((f_{1}-f_{2})(\overline{f_{1}}+
\overline{f_{2}}))\|_{W}\|f_{1}\|_{L^{2}\cap W}\\
&\phantom{pp}+C\|K*((\overline{f_{1}}-\overline{f_{2}})(f_{1}+f_{2}))\|_{W}\|f_{1}\|_{L^{2}\cap W}\\
&\daw \|\mathcal{K}_{1}\|_{L^{1}} \|(f_{1}-f_{2})(\overline{f_{1}}+
\overline{f_{2}})\|_{L^{1}} \|f_{1}\|_{L^{2}\cap W}\\
&\phantom{pp}+C\|\mathcal{K}_{2}\|_{L^{\infty}} \|(f_{1}-f_{2})(\overline{f_{1}}+
\overline{f_{2}})\|_{W} \|f_{1}\|_{L^{2}\cap W}\\
&\phantom{pp}+ C\|\mathcal{K}_{1}\|_{L^{1}} \|(\overline{f_{1}}-\overline{f_{2}})(f_{1}+f_{2})\|_{L^{1}}\|f_{1}\|_{L^{2}\cap W}\\
&\phantom{pp}+ C\|\mathcal{K}_{2}\|_{L^{\infty}} \|(\overline{f_{1}}-\overline{f_{2}})(f_{1}+f_{2})\|_{W}\|f_{1}\|_{L^{2}\cap W}\\
&\daw
C\|f_{1}\|^{2}_{L^{2}\cap W}\|f_{1}-f_{2}\|_{L^{2}\cap W}
+C\|f_{1}\|_{L^{2}\cap W}\|f_{2}\|_{L^{2}\cap W}\|f_{1}-f_{2}\|_{L^{2}\cap W}\\
&\phantom{pp}+C\|f_{2}\|^{2}_{L^{2}\cap W}\|f_{1}-f_{2}\|_{L^{2}\cap W}\\
&\daw C (\|f_{1}\|^{2}_{L^{2}\cap W}+
\|f_{2}\|^{2}_{L^{2}\cap W})\|f_{1}-f_{2}\|_{L^{2}\cap W},
\end{aligned}
\]
where we have used the Cauchy-Schwartz inequality. The desired control follows easily.
\end{proof}
Let $T>0$ to be specified later. We set

\[
X= \lbrace u \in C([0,T]; L^{2} \cap W),
\|u\|_{L^{\infty}([0,T]; L^{2}\cap W)} \daw 2\|u_{0}\|_{L^{2}\cap W} \rbrace.
\]

Duhamel's formulation of \eqref{eq:(SP')} reads

\[
u(t)= U(t)u_{0} -i\lambda \int_{0}^{t} U(t-\tau)(K*|u|^{2})u(\tau)\; \mathrm{d}\tau.
\]
We denote by $\Phi(u)(t)$ the right hand side in the above formula and $G(u)=\Phi(u)(t)- U(t)u_{0}$. For $q\geq 1$ and a space $S$
we define  $\|u\|_{L^{q}_{T}S}:= \|u\|_{L^{q}([0,T]; S)}$ for $u\in L^{q}([0,T]; S)$. Let $ u\in X $. We have
\[
\begin{aligned}
\|\Phi(u)\|_{L^{\infty}_{T}L^{2}\cap L^{\infty}} & \daw
\|\Phi(u)(t)\|_{L^{\infty}_{T}L^{2}} + \|\Phi(u)(t)\|_{L^{\infty}_{T}W},
\end{aligned}
\]
and
\[
\begin{aligned}
\|\Phi(u)\|_{L^{\infty}_{T}L^{2}} &\daw \|u_{0}\|_{L^{2}} +
\|G(u)\|_{L^{\infty}_{T}L^{2}}\\
& \daw \|u_{0}\|_{L^{2}}+ C\|g\|_{L^{1}_{T}L^{2}}\\
&\daw \|u_{0}\|_{L^{2}}+ C\|K*|u|^{2}\|_{L^{\infty}_{T}L^{\infty}} \|u\|_{L^{\infty}_{T}L^{2}} T.
\end{aligned}
\]
We apply Lemma \ref{lemm_W} after replacing $h$ by $|u|^{2}$.
We obtain by the embedding of $W$ into $L^{\infty}$
\[
\begin{aligned}
\|K*|u|^{2}\|_{L^{\infty}} &\daw C \|K*|u|^{2}\|_{W}\\
&\daw C\|\mathcal{K}_{1}\|_{L^{1}}\||u|^{2}\|_{W} + C\|\mathcal{K}_{2}\|_{L^{\infty}} \||u|^{2}\|_{L^{1}}\\
&\daw C\|u\|^{2}_{L^{2}\cap W}.
\end{aligned}
\]
So we have
\[
\begin{aligned}
\|\Phi(u)\|_{L^{\infty}_{T}L^{2}}
&\daw \|u_{0}\|_{L^{2}}+ C\|u\|^{2}_{L^{\infty}_{T}L^{2}\cap W} \|u\|_{L^{\infty}_{T}L^{2}} T\\
&\daw \|u_{0}\|_{L^{2}}+ C\|u\|^{3}_{L^{\infty}_{T}L^{2}\cap W} T.
\end{aligned}
\]
Moreover, always by applying Lemma \ref{lemm_W} with $h=|u|^{2}$ we obtain
\[
\begin{aligned}
\|\Phi(u)(t)\|_{W}& \daw \|u_{0}\|_{W}+ \int_0^t \|K*|u|^{2}(\tau)\|_{W}\|u(\tau)\|_{W} \, \mathrm{d}\tau\\
&\daw \|u_{0}\|_{W}+ C\int_0^t  \|u(\tau)\|^{3}_{L^{2}\cap W}\, \mathrm{d} \tau,\\
\end{aligned}
\]
and thus
\[
\begin{aligned}
\|\Phi(u)\|_{L^{\infty}_{T}W} &\daw \|u_{0}\|_{W}+ C \|u\|^{3}_{L^{\infty}_{T}L^{2}\cap W}T.\\
\end{aligned}
\]
Finally we have
\[
\|\Phi(u)\|_{L^{\infty}_{T}L^{2}\cap W} \daw  \|u_{0}\|_{L^{2}\cap W}
+ C\|u_{0}\|_{L^{2}\cap W}^{3}T.
\]
By reducing sufficiently $T$ (depending on $\|u_{0}\|_{L^{2}\cap W}$)
we get for all $u\in X$, $ \|\Phi(u)(t)\|_{L^{\infty}_{T}L^{2}} \daw 2 \|u_{0}\|_{L^{2}\cap W}$.
Moreover, for $u, v \in X$ we have
\[
\|\Phi(u)-\Phi(v)\|_{L^{\infty}_{T}L^{2}\cap W} \daw \|G(u)-G(v)\|_{L^{\infty}_{T}L^{2}}
+ \|G(u)-G(v)\|_{L^{\infty}_{T}W}.
\]
For some positive constant $C$ we have
\[
\begin{aligned}
\|G(u)-G(v)\|_{L^{\infty}_{T}L^{2}} &\daw C\|g(u)-g(v)\|_{L^{1}_{T}L^{2}}\\
&\daw C\|g(u)-g(v)\|_{L^{\infty}_{T}L^{2}}T\\
&\daw C\|g(u)-g(v)\|_{L^{\infty}_{T}L^{2}\cap W}T.\\
\end{aligned}
\]
Moreover
\[
\begin{aligned}
\|G(u)-G(v)\|_{W} &\daw \int_0^t \|g(u)(\tau)-g(v)(\tau)\|_{W} \, \mathrm{d}\tau,\\
\end{aligned}
\]
and
\[
\begin{aligned}
\|G(u)-G(v)\|_{L^{\infty}_{T} W} &\daw \|g(u)-g(v)\|_{L^{\infty}_{T}W}T\\
& \daw \|g(u)-g(v)\|_{L^{\infty}_{T}L^{2}\cap W}T.
\end{aligned}
\]
By replacing $f_{1}, f_{2}$ by $u, v$ respectively in Lemma \ref{lemm_W} we obtain
\[
\begin{aligned}
\|\Phi(u)-\Phi(v)\|_{L^{\infty}_{T}L^{2}\cap W}  &\daw C\|g(u)-g(v)\|_{L^{\infty}_{T}L^{2}\cap W}T\\
&\daw C \|u_{0}\|_{L^{2}\cap W}^{2} \|u-v\|_{L^{\infty}_{T}L^{2}\cap W}T.
\end{aligned}
\]
Choosing $T$ possibly smaller (still depending on $\|u_{0}\|_{L^{2}\cap W}$) we deduce that
$\Phi$ is a contraction from $X$ to $X$. Thus $\Phi$ has a unique fixed point $u\in X$ and Theorem
$\ref{exist_uniq}$ follows.
\end{proof}

\section{Derivation of the approximate solution}

We consider the rescaled version of \eqref{eq:(SP')} :
\begin{equation}
\label{eq:SPultim}
\ i\varepsilon \partial_{t} u^{\varepsilon} +
\frac{\eps^{2}}{2} \Delta u^{\varepsilon}=
\varepsilon \lambda (K*{|u^{\varepsilon}|}^{2})u^{\varepsilon}.
\end{equation}
We seek an approximation of solutions to \eqref{eq:SPultim} in the form
\[
u^{\eps}_{app}(t,x)=\sum_{j\in \N} a_{j}(t,x) e^{i\phi_{j}(t,x)/\eps}.
\]
We begin by proceeding formally. We plug the ansatz above into \eqref{eq:SPultim}. This yields
\begin{equation}
\label{eq_sol_app}
i\eps\partial_{t}u^{\eps}_{app}+\frac{\eps^{2}}{2} \Delta u^{\eps}_{app}-
\eps(K*|u^{\eps}_{app}|^{2})u^{\eps}_{app}= \sum_{k=0}^{2} \eps^{k}Z^{\eps}_{k}
+\eps(W^{\eps}+r^{\eps}),
\end{equation}
with
\[
Z^{\eps}_{0}= -\sum_{j\in \N}(\partial_{t}\phi_{j}+\frac{1}{2}|\nabla\phi_{j}|^{2})a_{j}e^{i\phi_{j}/\eps},
\]
\[
Z^{\eps}_{1}= i\sum_{j\in \N} (\partial_{t}a_{j}+\nabla\phi_{j}\nabla a_{j}+
\frac{1}{2}a_{j}\Delta\phi_{j})e^{i\phi_{j}/\eps},
\]
and
\begin{equation}
\label{term_r2}
Z^{\eps}_{2}=\frac{1}{2}\sum_{j\in \N} \Delta a_{j}e^{i\phi_{j}/\eps}.
\end{equation}
Other terms appear due to the presence of the nonlinearity which are
\begin{equation}
\label{pas_rest}
W^{\eps}= -\sum_{j\in \N}(K*\sum_{\ell\in \N}|a_{\ell}|^{2})a_{j}e^{i\phi_{j}/\eps},
\end{equation}
and
\begin{equation}
\label{rest}
r^{\eps}= -\sum_{j\in \N}(K*\sum_{\tiny \begin{matrix}k,\ell \in \N \\
k\neq \ell\end{matrix}} (a_{k}\overline{a_{\ell}}e^{i(\phi_{k}-\phi_{\ell})/\eps}))a_{j}e^{i\phi_{j}/\eps}.
\end{equation}
We aim to eliminate all equal powers of $\eps$. Hence, by setting $Z^{\eps}_{0}=0$ we obtain
the eikonal equation, whose solution is explicitly given by
\begin{equation}
\label{eikonale}
\phi_{j}(t,x)= \kappa_{j}\cdot x-\frac{t}{2} |\kappa_{j}|^{2}.
\end{equation}
Next, we set $Z^{\eps}_{1}+W^{\eps}=0$ without including $r^{\eps}$. This latter will constitute the first term error, the second
will be $Z^{\eps}_{2}$. We obtain for all $j\in\N$

\begin{equation}
\label{eq_amplitude}
\partial_{t}a_{j}+ \kappa_{j}\cdot \nabla a_{j} = -i(K*\sum_{\ell\in \N}|a_{\ell}|^{2})a_{j}, \quad
a_{j}(0)=\alpha_{j},
\end{equation}

where we have used the fact that $\Delta\phi_{j}=0.$

\begin{lemme}
\label{lemm_ampl}
The transport equation \eqref{eq_amplitude} with initial amplitudes $(\alpha_{j})_{j\in \N} \in E(\R^{d})$ admits
a unique global-in-time solutions $a=(a_{j})_{j\in \N}\in C([0,\infty[;E(\R^{d}))$, which can be written
in the form
\begin{equation}
\label{form_amplitude}
a_{j}(t,x)=\alpha_{j}(x-t\kappa_{j})e^{iS_{j}(t,x)},
\end{equation}
where
\begin{equation}
\label{ampli_puissance}
S_{j}(t,x)= -\int_{0}^{t}(K*\sum_{\ell \in \N} |\alpha_{\ell}(x +(\tau-t)\kappa_{j}-\tau\kappa_{\ell})|^{2})\;\mathrm{d}\tau.
\end{equation}
\end{lemme}

\begin{proof}
We multiply \eqref{eq_amplitude} with $\overline{a_{j}}$. We obtain
\begin{equation}
\label{dd}
\overline{a_{j}}\partial_{t}a_{j}+ \kappa_{j}\cdot (\overline{a_{j}}\nabla a_{j}) = -i(K*\sum_{\ell\in \N}|a_{\ell}|^{2})|a_{j}|^{2}
\in i\R.
\end{equation}
But
\[
2\mathrm{Re}(\overline{a_{j}}\partial_{t}a_{j}+ \kappa_{j}\cdot (\overline{a_{j}}\nabla a_{j}))
=(\partial_{t}+\kappa_{j}\cdot \nabla)|a_{j}|^{2}.
\]
So, we deduce from \eqref{dd} that $(\partial_{t}+\kappa_{j}\cdot \nabla)|a_{j}|^{2}=0$, which yields
\[
|a_{j}(t,x)|^{2}=|\alpha_{j}(x-t\kappa_{j})|^{2}.
\]
and \eqref{form_amplitude} follows for some real function $S_{j}$. To determine $S_{j}$ we inject
\eqref{form_amplitude} into \eqref{eq_amplitude}. We get
\[
i(\partial_{t}+\kappa_{j}\cdot\nabla)S_{j}(t,x)\alpha_{j}(x-t\kappa_{j})=-i(K*\sum_{\ell\in \N}|\alpha_{j}(x-t\kappa_{\ell})|^{2})
\alpha_{j}(x-t\kappa_{j}).
\]
It suffices to impose
\[
\partial_{t}(S_{j}(t,x+t\kappa_{j}))=-K*\sum_{\ell\in \N}|\alpha_{j}(x-t(\kappa_{\ell}-\kappa_{j}))|^{2},
\]
which yields
\[
S_{j}(t,x+t\kappa_{j})=-\int_{0}^{t} K*\sum_{\ell\in \N}|\alpha_{j}(x-\tau(\kappa_{\ell}-\kappa_{j}))|^{2}\;\mathrm{d}\tau,
\]
and finally we have
\[
S_{j}(t,x)=-\int_{0}^{t} K*\sum_{\ell\in \N}|\alpha_{j}(x+(\tau-t)\kappa_{j})-\tau\kappa_{\ell}|^{2}\;\mathrm{d}\tau.
\]

Global in time existence of $a_{j}$'s is not trivial from their explicit formula like we can see in the following. We rewrite \eqref{eq_amplitude} in its integral form

\begin{equation}
\label{form_integ_aj}
a_{j}(t,x)= \alpha_{j}(x-t\kappa_{j}) + \int_{0}^{t} \mathcal{N}(a)_{j}\, (\tau, x+(\tau-t)\kappa_{j})\mathrm{d}\tau,
\end{equation}
where the nonlinearity $\mathcal{N}$ is given by
\[
\mathcal{N}(a)_{j}=-i\bigg(K*\sum_{\ell\in \N} |a_{\ell}|^{2}\bigg)a_{j}.
\]
For $a\in E(\R^{d})$ we have by definition
\[
\begin{aligned}
\|\mathcal{N}(a)\|_{E} &= \sum_{j\in \N} \|\mathcal{N}(a)_{j}\|_{Y}
=\sum_{j\in \N} \sum_{|\eta|\daw n}\|\partial^{\eta}\mathcal{N}(a)_{j}\|_{L^{2}\cap W}.
\end{aligned}
\]
Let $|\eta|\daw n$. Leibnitz formula yields
\[
\begin{aligned}
\|\partial^{\eta}\mathcal{N}(a)_{j}\|_{L^{2}\cap W}
=&\|\partial^{\eta}\bigg((K*\sum_{\ell\in \N} |a_{\ell}|^{2})a_{j}\bigg)\|_{L^{2}\cap W}\\
\daw &\sum_{\theta\daw \eta} C^{\theta}_{\eta} \|(K*\sum_{\ell \in \N} \partial^{\theta}|a_{\ell}|^{2})\;\partial^{\eta-\theta}a_{j}\|_{L^{2}\cap W}\\
\daw &C \sum_{|\theta|\daw n} \|(K*\sum_{\ell \in \N} \partial^{\theta}|a_{\ell}|^{2})\|_{W}\|a_{j}\|_{Y}\\
\daw &C \|a_{j}\|_{Y} \bigg(\sum_{\ell\in \N}\|\mathcal{K}_{1}\|_{L^{1}}
\sum_{|\theta|\daw n}\|\partial^{\theta}|a_{\ell}|^{2}\|_{L^{1}}\\
&+\|\mathcal{K}_{2}\|_{L^{\infty}}
\sum_{|\theta|\daw n}\|\partial^{\theta}|a_{\ell}|^{2}\|_{W}\bigg)\\
\daw &C \|a_{j}\|_{Y}\sum_{\ell\in \N}(\|\mathcal{K}_{1}\|_{L^{1}}+\|\mathcal{K}_{2}\|_{L^{\infty}}) \|a_{\ell}\|^{2}_{Y}\\
\daw &C\|a_{j}\|_{Y}\|a\|^{2}_{E}.
\end{aligned}
\]
Finally we have
\[
\|\mathcal{N}(a)\|_{E}\daw C\|a\|^{3}_{E}.
\]
This shows that $\mathcal{N}(a)$ defines a continuous mapping from $E^{3}$ to $E$ and by the standard
Cauchy-Lipschitz theorem for the ordinary differential equations, a local-in-time existence results immediately follows.
Now, we have to show that the solution $a(t)=(a_{j}(t))_{j\in \N} $ is global in time. Let $[0,T_{\max}[$
be the maximal time interval where $(a_{j})_{j\in \N}$ is defined. From \eqref{form_integ_aj} we have for $|\eta|\daw n$ and $t\in [0,T_{\max}[$
\[
\begin{aligned}
\partial^{\eta}a_{j}(t)=&\;\partial^{\eta}\alpha_{j}(\cdot-t\kappa_{j}) +
\int_{0}^{t} \partial^\eta\mathcal{N}(a)_{j}(\tau,\cdot+(\tau-t)\kappa_{j})\;\mathrm{d}\tau\\
=&\;\partial^{\eta}\alpha_{j}(\cdot-t\kappa_{j})\\
&+\int_{0}^{t}\sum_{\theta\daw \eta}C^{\theta}_{\eta} (K*\sum_{\ell \in \N} \partial^{\theta}|\alpha_{\ell}(\cdot-t\kappa_{j})|^{2})\;\partial^{\eta-\theta}a_{j}(\tau,\cdot+(\tau-t)\kappa_{j})\;\mathrm{d}\tau.\\
\end{aligned}
\]
Taking $L^{2}\cap W$ norm yields
\[
\|\partial^{\eta}a_{j}(t)\|_{L^{2}\cap W} \daw \phantom{\|\partial^{\eta}\alpha_{j}\|_{L^{2}\cap W}+C\sum_{\theta\daw \eta} \int_{0}^{t}
\|K*\sum_{\ell \in \N} \partial^{\theta}|\alpha_{\ell}(\cdot-t\kappa_{j})|^{2}\|_{W}
 \|\partial^{\eta-\theta}a_{j}(\tau)\|_{L^{2}\cap W}\;\mathrm{d}\tau}
 \]
 \[
 \begin{aligned}
 &\daw \|\partial^{\eta}\alpha_{j}\|_{L^{2}\cap W}+ C\sum_{\theta\daw \eta} \int_{0}^{t}
 \|K*\sum_{\ell \in \N} \partial^{\theta}|\alpha_{\ell}(\cdot-t\kappa_{j})|^{2}\|_{W}
  \|\partial^{\eta-\theta}a_{j}(\tau)\|_{L^{2}\cap W}\;\mathrm{d}\tau\\
  &\daw\|\partial^{\eta}\alpha_{j}\|_{L^{2}\cap W}
  +C\sum_{\theta\daw \eta} \int_{0}^{t} \|K*\sum_{\ell \in \N} \partial^{\theta}|\alpha_{\ell}|^{2}\|_{W}
  \|\partial^{\eta-\theta}a_{j}(\tau)\|_{L^{2}\cap W}\;\mathrm{d}\tau\\
  &\daw \|\partial^{\eta}\alpha_{j}\|_{L^{2}\cap W}
  +C (\|\mathcal{K}_{1}\|_{L^{1}}+\|\mathcal{K}_{2}\|_{L^{\infty}})\sum_{\ell\in \N}\|\alpha_{\ell}\|^{2}_{Y} \int_{0}^{t}\|a_{j}(\tau)\|_{Y}\;\mathrm{d}\tau\\
  &\daw\|\partial^{\eta}\alpha_{j}\|_{L^{2}\cap W}+C\int_{0}^{t}\|a_{j}(\tau)\|_{Y}\;\mathrm{d}\tau,\\
  \end{aligned}
\]
and so
\[
\begin{aligned}
\|a_{j}(t)\|_{Y}=\sum_{|\eta|\daw n} \|\partial^{\eta} a_{j}(t)\|_{L^{2}\cap W}
\daw \|\alpha_{j}\|_{Y} +C\int_{0}^{t}\|a_{j}(\tau)\|_{Y}\;\mathrm{d}\tau.
\end{aligned}
\]
By Gronwall lemma
\[
\|a_{j}(t)\|_{Y}\daw \|\alpha_{j}\|_{Y}e^{Ct}.
\]
for some positive constant $C$ independent of $j$ and $t$. After summing with respect to $j$ we obtain
\[
\|a(t)\|_{E}\daw \|\alpha\|_{E}e^{Ct}.
\]
The growth of $\|a(t)\|_{E}$ is at most exponential, and thus can not explode at finite time. We deduce that
 $T_{\max}=\infty$ and $a=(a_{j})_{j\in \N}\in C([0,\infty[;E(\R^{d}))$.

\end{proof}
\newpage
\section{Estimations on the remainder}

We will estimate in $ L^{2}\cap W$ the term $r^{\eps}$ and $Z^{\eps}_{2}$. To this end, we assume
$(\alpha_{j})_{j\in \N} \in \ell^{1}(\N, Y(\R^{d}))$.

\begin{Prop}
\label{estim_error}
Let $r^{\eps}$ be defined by \eqref{rest} with the plane-wave phases $\phi_{j}$
given by \eqref{eikonale}. We have the following bound:

\begin{equation}
\label{control_R}
\|r^{\eps}\|_{L^{2}\cap W}\daw C\delta^{\gamma-d}\|a\|_{E}^{3}\eps^{d-\gamma},
\end{equation}

\begin{equation}
\label{control_X2}
\|Z^{\eps}_{2}\|_{L^{2}\cap W} \daw \|a\|_{E},
\end{equation}
where $\delta$ is like defined in Theorem \ref{mainTh}.

\end{Prop}

\begin{proof}

Inequality \eqref{control_X2} is obvious. We recall the definition of $r^{\eps}$:

\[
r^{\eps}= -\sum_{j\in \N}(K*\sum_{\tiny \begin{matrix}k,\ell \in \N \\
k\neq \ell\end{matrix}} a_{k}\overline{a_{\ell}}e^{i(\phi_{k}-\phi_{\ell})/\eps})a_{j}e^{i\phi_{j}/\eps}.
\]

We estimate $W$ norm of $r^{\eps}$, we obtain

\[
\begin{aligned}
\|r^{\eps}\|_{W} &= \|\sum_{j\in \N}(K*\sum_{\tiny \begin{matrix}k,\ell \in \N \\
k\neq \ell\end{matrix}} a_{k}\overline{a_{\ell}}e^{i(\phi_{k}-\phi_{\ell})/\eps})
a_{j}e^{i\phi_{j}/\eps}\|_{W}\\
 &\daw \sum_{j\in \N}\|K*\sum_{\tiny \begin{matrix}k,\ell \in \N \\
k\neq \ell\end{matrix}} a_{k}\overline{a_{\ell}}e^{i(\phi_{k}-\phi_{\ell})/\eps}\|_{W}\|a_{j}\|_{W} \\
&\daw \sum_{\tiny \begin{matrix}k,\ell \in \N \\
k\neq \ell\end{matrix}} \|K*a_{k}\overline{a_{\ell}}e^{i(\phi_{k}-\phi_{\ell})/\eps}\|_{W}
\sum_{j\in \N} \|a_{j}\|_{W}.
\end{aligned}
\]

Let us look more closely to the $W$ norm of the convolution above.
For $b$ sufficiently regular and $\omega\in \R^{d}$, we denote $I^{\eps}(x)= K*(be^{i\omega\cdot x/\eps})$.
Remark that

\[
\begin{aligned}
\chap{I^{\eps}}& = (2\pi)^{d/2} \chap{K}\;\chap{be^{i\omega\cdot x/\eps}}\\
&= (2\pi)^{d/2}\chap{K}\;\chap{b}(\cdot-\frac{\omega}{\eps}).
\end{aligned}
\]

So, if we take the $W$ norm of $I^{\eps}$ we obtain
\[
\begin{aligned}
\|I^{\eps}\|_{W}=&\; \|\chap{I^{\eps}}\|_{L^{1}}\\
=&\;C \|\chap{K}\, \lang\xi - \frac{\omega}{\eps}\rang^{m}\lang\xi - \frac{\omega}{\eps}\rang^{-m} \chap{b}(\cdot -\frac{\omega}{\eps})\|_{L^{1}}\\
\daw &\;C\|\mathcal{K}_{1}\,\lang\xi - \frac{\omega}{\eps}\rang^{m}\lang\xi - \frac{\omega}{\eps}\rang^{-m}\chap{b}(\cdot -\frac{\omega}{\eps})\|_{L^{1}}\\
&+\;C\|\mathcal{K}_{2}\,\lang\xi - \frac{\omega}{\eps}\rang^{m}\lang\xi - \frac{\omega}{\eps}\rang^{-m}\chap{b}(\cdot -\frac{\omega}{\eps})\|_{L^{1}} \\
\daw &\;C\|\mathcal{K}_{1}\|_{L^{1}}\|\lang\xi - \frac{\omega}{\eps}\rang^{m}\chap{b}(\cdot -\frac{\omega}{\eps})\|_{L^{\infty}}
\sup_{|\xi|\daw 1}\lang\xi-\frac{\omega}{\eps}\rang^{-m}\\
&+\;C\|\lang\xi - \frac{\omega}{\eps}\rang^{m}\chap{b}(\cdot -\frac{\omega}{\eps})\|_{L^{1}}
\sup_{|\xi|> 1}\bigg(|\xi|^{\gamma-d}\lang\xi-\frac{\omega}{\eps}\rang^{-m}\bigg).\\
\end{aligned}
\]

By Peetre inequality we have for all $m\daw d-\gamma$ and $|\xi| > 1$

\[
\begin{aligned}
|\xi|^{d-\gamma}\lang\xi - \frac{\omega}{\eps}\rang^{m} \geqslant |\xi|^{d-\gamma}\frac{\lang\frac{\omega}{\eps}\rang^{m}}{\lang\xi\rang^{m}}
\geqslant C\lang\frac{\omega}{\eps}\rang^{m}\geqslant C\eps^{-m},
\end{aligned}
\]
and thus
\[
\sup_{|\xi|> 1}(|\xi|^{\gamma-d}\lang\xi-\frac{\omega}{\eps}\rang^{-m}) \daw C\eps^{m}\daw C\eps^{d-\gamma},
\]
Moreover
\[
\begin{aligned}
\lang\xi -\frac{\omega}{\eps}\rang^{m} \geqslant \frac{\lang\frac{\omega}{\eps}\rang^{m}}{\lang\xi\rang^{m}}
\geqslant \lang\frac{\omega}{\eps}\rang^{m}
\geqslant C \eps^{-m},
\end{aligned}
\]
and
\[
\sup_{|\xi|\daw 1}\lang\xi-\frac{\omega}{\eps}\rang^{-m}\daw C\eps^{d-\gamma}.
\]
So, for $m=d-\gamma$
\[
\begin{aligned}
\label{inegI}
\|I^{\eps}\|_{W} \daw C\bigg(\|\mathcal{K}_{1}\|_{L^{1}} \|\lang\xi\rang^{d-\gamma}\chap{b}\|_{L^{\infty}} + C\|\mathcal{K}_{2}\|_{L^{\infty}}\|\lang\xi\rang\chap{b}\|_{L^{1}}\bigg) |\omega|^{d-\gamma}\eps^{d-\gamma}.
\end{aligned}
\]

Hence, for all $k,\ell \in \N, k\neq \ell$

\[
\begin{aligned}
\|K*a_{k}\overline{a_{\ell}}e^{\frac{\phi_{k}-\phi_{\ell}}{\eps}}\|_{W}
\daw &\;C\|\mathcal{K}_{1}\|_{L^{1}} \|\lang\xi\rang^{d-\gamma}\chap{a_{k}\overline{a_{\ell}}}\|_{L^{\infty}}|\kappa_{k}-\kappa_{\ell}|^{d-\gamma}\eps^{d-\gamma}\\
&+ C\|\mathcal{K}_{2}\|_{L^{\infty}}\|\lang\xi\rang^{d-\gamma}\chap{a_{k}\overline{a_{\ell}}}\|_{L^{1}} |\kappa_{k}-\kappa_{\ell}|^{d-\gamma}\eps^{d-\gamma}\\
\daw &\;C\delta^{\gamma-d}\|\mathcal{K}_{1}\|_{L^{1}} \|\lang\xi\rang^{d-\gamma}\chap{a_{k}\overline{a_{\ell}}}\|_{L^{\infty}}\eps^{d-\gamma}\\
&+\; C\delta^{\gamma-d}\|\mathcal{K}_{2}\|_{L^{\infty}}\|\lang\xi\rang^{d-\gamma}\chap{a_{k}\overline{a_{\ell}}}\|_{L^{1}} \eps^{d-\gamma}.\\
\end{aligned}
\]
Remark that
\[
\begin{aligned}
\|\lang\xi\rang^{d-\gamma}\chap{a_{k}\overline{a_{\ell}}}\|_{L^{\infty}}
&\daw \|\lang\xi\rang^{n}\chap{a_{k}\overline{a_{\ell}}}\|_{L^{\infty}}\\
&\daw C\sum_{|\eta|\daw n} \|\chap{\partial^{\eta} (a_{k}\overline{a_{\ell}})}\|_{L^{\infty}}\\
&\daw C\sum_{|\eta|\daw n} \|\partial^{\eta} (a_{k}\overline{a_{\ell}})\|_{L^{1}}\\
&\daw C\sum_{|\eta|\daw n} \|\partial^{\eta} a_{k}\|_{L^{2}}\sum_{|\eta|\daw n} \|\partial^{\eta} a_{\ell}\|_{L^{2}}\\
&\daw C \|a_{k}\|_{Y} \|a_{\ell}\|_{Y}.
\end{aligned}
\]
Moreover
\[
\begin{aligned}
\|\lang\xi\rang^{d-\gamma}\chap{a_{k}\overline{a_{\ell}}}\|_{L^{1}}
&\daw \|\lang\xi\rang^{n}\chap{a_{k}\overline{a_{\ell}}}\|_{L^{1}}\\
&\daw C\sum_{|\eta|\daw n} \|\chap{\partial^{\eta} (a_{k}\overline{a_{\ell}})}\|_{L^{1}}\\
&\daw C\sum_{|\eta|\daw n} \|\partial^{\eta} (a_{k}\overline{a_{\ell}})\|_{W}\\
&\daw C\sum_{|\eta|\daw n} \|\partial^{\eta} a_{k}\|_{W}\sum_{|\eta|\daw n} \|\partial^{\eta} a_{\ell}\|_{W}\\
&\daw C \|a_{k}\|_{Y} \|a_{\ell}\|_{Y}.
\end{aligned}
\]

Thus we obtain

\[
\begin{aligned}
\|K*a_{k}\overline{a_{\ell}}e^{\frac{\phi_{k}-\phi_{\ell}}{\eps}}\|_{W}
& \daw C\delta^{\gamma-d}(\|\mathcal{K}_{1}\|_{L^{1}}
+ C\|\mathcal{K}_{2}\|_{L^{\infty}})\|a_{k}\|_{Y} \|a_{\ell}\|_{Y}\eps^{d-\gamma}.\\
\end{aligned}
\]

In view of the above inequality, we deduce
\[
\begin{aligned}
\|r^{\eps}\|_{W}
&\daw \sum_{\tiny \begin{matrix}k,\ell \in \N \\
k\neq \ell\end{matrix}} C\delta^{\gamma-d}(\|\mathcal{K}_{1}\|_{L^{1}}
+ C\|\mathcal{K}_{2}\|_{L^{\infty}})\|a_{k}\|_{Y} \|a_{\ell}\|_{Y}\eps^{d-\gamma}
\sum_{j\in \N} \|a_{j}\|_{Y}\\
&\daw C\delta^{\gamma-d}(\|\mathcal{K}_{1}\|_{L^{1}}
+ C\|\mathcal{K}_{2}\|_{L^{\infty}})\sum_{k\in \N} \|a_{k}\|_{Y}\sum_{\ell\in \N} \|a_{\ell}\|_{Y}
\sum_{j\in \N} \|a_{j}\|_{Y}\eps^{d-\gamma}\\
&\daw C\delta^{\gamma-d}(\|\mathcal{K}_{1}\|_{L^{1}}
+ C\|\mathcal{K}_{2}\|_{L^{\infty}})\|a\|_{E}^{3}\eps^{d-\gamma}.
\end{aligned}
\]

To control $L^{2}$ norm of $r^{\eps}$ it suffices to remark that

\[
\begin{aligned}
\|r^{\eps}\|_{L^{2}} &= \|\sum_{j\in \N}(K*\sum_{\tiny \begin{matrix}k,\ell \in \N \\
k\neq \ell\end{matrix}} a_{k}\overline{a_{\ell}}e^{i(\phi_{k}-\phi_{\ell})/\eps})
a_{j}e^{i\frac{\phi_{j}}{\eps}}\|_{L^{2}}\\
&\daw \sum_{j\in \N}\|K*\sum_{\tiny \begin{matrix}k,\ell \in \N \\
k\neq \ell\end{matrix}} a_{k}\overline{a_{\ell}}e^{i(\phi_{k}-\phi_{\ell})/\eps}\|_{L^{\infty}}\|a_{j}\|_{L^{2}} \\
&\daw C\sum_{\tiny \begin{matrix}k,\ell \in \N \\
k\neq \ell\end{matrix}} \|K*a_{k}\overline{a_{\ell}}e^{i(\phi_{k}-\phi_{\ell})/\eps}\|_{W}
\sum_{j\in \N} \|a_{j}\|_{Y}.
\end{aligned}
\]
and in the same way as previously we get

\[
\begin{aligned}
\|r^{\eps}\|_{L^{2}} & \daw C\delta^{\gamma-d}(\|\mathcal{K}_{1}\|_{L^{1}}
+ \|\mathcal{K}_{2}\|_{L^{\infty}})\|a\|_{E}^{3}\eps^{d-\gamma}.
\end{aligned}
\]
Finally
\[
\begin{aligned}
\|r^{\eps}\|_{L^{2}\cap W} & \daw C\delta^{\gamma-d}(\|\mathcal{K}_{1}\|_{L^{1}}
+ \|\mathcal{K}_{2}\|_{L^{\infty}})\|a\|_{E}^{3}\eps^{d-\gamma}.
\end{aligned}
\]

\end{proof}

\section{Justification of the approach}
In this section we show that the solution $u^{\eps}_{app}$	is a
good approximation of the exact solution in the sense of Theorem \ref{mainTh}.  We assume
that assumptions of Theorem \ref{mainTh} are verified. At this stage
we have constructed an approximate solution
\[
u^{\eps}_{app} = \sum_{j\in \N} a_{j}(t,x)e^{i\phi(t,x)/\eps},
\]
in $C([0,\infty[;E(\R^{d}))$ where the corresponding $\phi_{j}$'s are like constructed in \eqref{eikonale} and the profiles
are given by Lemma \ref{lemm_ampl}. Let $T^{\eps}>0$ the time of local existence of the exact solution $u^{\eps}$ to
\eqref{eq:SPultim} given by the fixed point argument.
We consider a fixed $T>0$ and we introduce the error between the exact solution $u^{\eps}$ and the approximate solution $u^{\eps}_{app}$ :
 \[
w^{\eps}=u^{\eps}-u^{\eps}_{app}.
\]
We have in particular
\[
u^{\eps}_{app} \in C([0,T], L^{2}\cap W),
\]
so there exists $R>0$ independent of $\eps$ such that
\[
\|u^{\eps}_{app}(t)\|_{L^{2}\cap W}\daw R, \; \forall t \in [0,T].
\]
Since $u^{\eps} \in C([0,T^{\eps}],L^{2}\cap W)$ and $w^{\eps}(0)=0$, there exists $t^{\eps}$
such that
\begin{equation}
\label{so_long_as_w}
\|w^{\eps}(t)\|_{L^{2}\cap W}\daw R,
\end{equation}
for $t\in [0,t^{\eps}]$. So long as \eqref{so_long_as_w} holds, we infer
\[
\begin{aligned}
\|w^{\eps}(t)\|_{L^{2}\cap W}\daw &\;|\lambda|\int_{0}^{t} \|\left(g(u^{\eps}+w^{\eps})-g(u^{\eps}_{app})\right)(\tau)\|_{L^{2}\cap W}
\; \mathrm{d}\tau \\
&+ |\lambda|\eps \int_{0}^{t} \|Z^{\eps}_{2}(\tau)\|_{L^{2}\cap W} \;\mathrm{d}\tau
+\;|\lambda|\int_{0}^{t} \|R(\tau)\|_{L^{2}\cap W} \; \mathrm{d}\tau\\
\daw &\;C\int_{0}^{t} \|w^{\eps}(\tau)\|_{L^{2}\cap W}\;\mathrm{d}\tau+C\eps \int_{0}^{t} \mathrm{d}\tau
+ C\eps^{d-\gamma}\int_{0}^{t} \mathrm{d}\tau\\
\daw &\;C\int_{0}^{t} \|w^{\eps}(\tau)\|_{L^{2}\cap W}\;\mathrm{d}\tau + C\eps^{\beta}\int_{0}^{t} \mathrm{d}\tau,
\end{aligned}
\]
where we have used Lemma \ref{lemm_W} and Proposition \ref{estim_error}, with $\beta=\min \lbrace 1,d-\gamma\rbrace$.
Gronwall lemma implies that so long as \eqref{so_long_as_w} holds,
\[
\begin{aligned}
\|w^{\eps}(t)\|_{L^{2}\cap W} &\daw C(e^{Ct}-1) \eps^{\beta},
\end{aligned}
\]
for all $t\in [0,t^{\eps}]$. Choosing $\eps\in ]0,\eps_{0}]$ with $\eps_{0}$ sufficiently small,
we see that \eqref{so_long_as_w} remains true for $t\in [0,T]$ and Theorem \ref{mainTh} follows.

\end{flushleft}

\bibliographystyle{elsarticle-num}

							\end{document}